\documentclass[reqno]{amsart}
\usepackage{amssymb}
\usepackage{hyperref}

\usepackage[latin1]{inputenc}
\usepackage[T1]{fontenc}
\newtheorem{theorem}{Theorem}[section]
\newtheorem{lemma}[theorem]{Lemma}
\newtheorem{proposition}[theorem]{Proposition}
\newtheorem{corollary}[theorem]{Corollary}
\theoremstyle{definition}
\newtheorem{definition}[theorem]{Definition}
\newtheorem{example}[theorem]{Example}

\theoremstyle{remark}
\newtheorem{remark}[theorem]{Remark}

\numberwithin{equation}{section}

\begin{document}

\title{Codiskcyclic sets of operators on complex topological vector spaces}
\author{ Mohamed Amouch and Otmane Benchiheb}
\address{Mohamed Amouch and Otmane Benchiheb,
University Chouaib Doukkali.
Department of Mathematics, Faculty of science
Eljadida, Morocco}
\email{amouch.m@ucd.ac.ma}

\email{otmane.benchiheb@gmail.com}

\begin{abstract}
Let $X$ be a complex topological vector space and $L(X)$ the set of all continuous linear operators on $X.$
In this paper, we extend the notion of the codiskcyclicity of a single operator $T\in L(X)$ to a set of operators $\Gamma\subset L(X).$
We prove some results for codiskcyclic sets of operators and we establish a codiskcyclicity criterion.
As an application, we study the codiskcyclicity of $C_0$-semigroups of operators.
\end{abstract}


\maketitle

\section{Introduction and Preliminary}
Let $X$ be a complex topological vector space and $L(X)$ the set of all continuous linear operators on $X.$
By an operator, we always mean a continuous linear operator.

The most studied notion in the linear dynamics is that of hypercyclicity$:$
an operator $T$ acting on $X$ is called hypercyclic if there exits a vector $x \in X$ such that the orbit of $x$ under $T$;
$$Orb(T,x)=\{T^nx \mbox{ : } n \in\mathbb{N}\},$$
is dense in $X$, such a vector $x$ is called hypercyclic
for $T$. The set of all hypercyclic vectors for $T$ is denoted by $HC(T)$.
 Another important notion in the linear dynamics is that of supercyclicity which was introduced in \cite{HW}.
We say that $T$ is supercyclic if there exists $x \in X$ such that
$$\mathbb{C}Orb(T,x)=\{\alpha T^nx \mbox{ : }\alpha\in\mathbb{C}\mbox{, } n \in\mathbb{N}\},$$
 is dense in $X$. The vector $x$ is called a supercyclic
vector for $T$. We denote by $SC(T)$ the set of all supercyclic vectors.
For more information about hypercyclic and supercyclic operators, see \cite{Bayart Matheron,Grosse,Erdmann Peris}.

Another notion in the linear dynamics which was been studied by many authors is that of codiscyclicity$:$
an operator $T$ is called codiskcyclic if there is $x\in X$ such that the codisk orbit of $x$ under $T;$
$$\mathbb{U}Orb(T,x)=\{\alpha T^{n}x\mbox{ : }\alpha\in\mathbb{U}\mbox{, } n\geq0\},$$
 is  dense in $X$, when $\mathbb{U}:=\{\alpha\in\mathbb{C}\mbox{ : }\vert\alpha\vert\geq1\}$.
In this case, the vector $x$ is called a codiskcyclic vector for $T$.
The set of all codiskcyclic vectors for $T$ is denoted by $\mathbb{U}C(T)$.
In the case of a separable complex Banach space, an operator $T$ is codiskcyclic if and only if it is codisk transitive; that is for each pair $(U,V)$ of nonempty open sets there exist some $\alpha\in\mathbb{U}$ and some $n\geq0$ such that 
$\alpha T^{n}(U)\cap V\neq \emptyset.$
For a general overview of the codiskcyclicity, see \cite{LZ,LZ1,WZ, Zeana}.


Recently, some notions of the linear dynamical system were introduced for a set $\Gamma$ of operators instead
of a single operator $T,$ see \cite{AOD,AOC,AOH,AOS,AKH,AKH1}$:$
 A set $\Gamma$ of operators is called hypercyclic if there exists a vector $x$ in $X$ such that its orbit under $\Gamma$; 
 $Orb(\Gamma,x)=\{Tx \mbox{ : }T\in\Gamma\},$
 is a dense subset of $X$. If there exits a vector $x\in X$ such that 
$\mathbb{C}Orb(\Gamma,x)=\{\alpha Tx\mbox{ : }T\in\Gamma\mbox{, }\alpha\in\mathbb{C}\},$
is a dense subset of $X$, then $\Gamma$ is supercyclic.
If $x\in X$ is a vector such that
$\mbox{span}\{Orb(\Gamma,x)\}=\mbox{span}\{Tx \mbox{ : }T\in\Gamma\}$
is dense in $X$, then $\Gamma$ is cyclic. 
If there exists a vector $x\in X$ such that its disk orbit under $T$;
$\mathbb{D}Orb(\Gamma,x)=\{\alpha Tx\mbox{ : }T\in\Gamma\mbox{, }\alpha\in\mathbb{D}\},$
is a dense subset $X$, then $\Gamma$ is called a diskcyclic, when $\mathbb{D}$ is the unit closed disk.
 In each case,
the vector $x$ is called a hypercyclic, a supercyclic, a cyclic and a diskcyclic vector for $\Gamma$,
respectively.
 
In this paper, we continue the study of the dynamics of a set of operator by introducing the concept of codiskcyclicity for a set of operators.

In Section $2$, we introduce and study the codiskcyclicity for a set of operators.
In particular,
we show that the set of codiscyclic vectors of a set $\Gamma$ is a $G_\delta$ type  and we prove that codiskcyclicity is preserved under quasi-similarity.

In Section $3$, we extend the notion of codisk transitivity of a single operator to a set of operators.
We give the relation between this notion and the concept of codiskcyclic and we establish a codiskcyclic criterion.
 
 In Section $4$, we study the codiskcyclicity of a $C_0$-semigroup of operators.
We show that the codiskcyclicity and the codisk transitivity are equivalent
and we prove that a codiskcyclic $C_0$-semigroup of operators exists on $X$ if and only if dim$(X)=1$ or dim$(X)=\infty$.
 

\section{ Codiskcyclic Sets of Operators}
In the following definition, we introduce the notion of the codiskcyclicity of a set of operators instead of a single operator.
\begin{definition}
We say that $\Gamma$ is codiskcyclic if there exists $x\in X$ for which  the codisk orbit of $x$ under $\Gamma;$
$$\mathbb{U} Orb(\Gamma,x):=\{\alpha Tx \mbox{ : }\alpha\in\mathbb{U}\mbox{, }T\in\Gamma\},$$
 is dense in $X$. 
The vector $x$ is called a codiskcyclic vector for $\Gamma$.
The set of all codiskcyclic vectors for $\Gamma$ is denoted by $\mathbb{U}C(\Gamma)$.
\end{definition}
\begin{remark}
An operator $T$ is codiskcyclic  if and only if the set
$\Gamma=\{T^n\mbox{ : }n\geq0\}$
is codiskcyclic.
\end{remark}
\begin{example}\label{ex2}
Let $f$ be a nonzero linear form on a locally convex space $X$
and $D$ be a subset of $X$ such that the set
$\mathbb{U}D:=\{\alpha x\mbox{ : }\alpha\in\mathbb{U}\mbox{, }x\in D\} $
is a dense subset of $X$.
For all $x\in X$, let $T_x$ defined by
$T_xy= f(y)x,$ for all $y\in X.$
Put $\Gamma_f=\{T_x \mbox{ : }x\in D\}$ and let $y$ be a vector of $X$ such that $f(y)\neq0.$
Then
$\mathbb{U}Orb(\Gamma_f,y)=\{\alpha T_xy \mbox{ : }x\in D\mbox{, } \alpha\in\mathbb{U}\}=\{\alpha f(y)x \mbox{ : }x\in D\mbox{, } \alpha\in\mathbb{U}\}=\mathbb{U}D. $
Hence, $\Gamma_f$ is codiskcyclic.
\end{example}

A necessary condition for the codiskcyclicity is given by the following proposition.
\begin{proposition}\label{pr1}
Let $X$ be a complex normed space and $\Gamma$ a subset of $L(X)$.
If $x$ is a codiskcyclic vector for $\Gamma)$, then
$\sup\{\Vert\alpha Tx \Vert \mbox{ : }\alpha\in\mathbb{U}\mbox{, }T\in\Gamma \}=+\infty. $
\end{proposition}
\begin{proof}
Let $x\in\mathbb{U}C(\Gamma)$.
Assume that
$ \sup\{\Vert\alpha Tx \Vert \mbox{ : }\alpha\in\mathbb{U}\mbox{, }T\in\Gamma \}=m<+\infty,$
and let $y\in X$ such that $\Vert y \Vert>m$. Since $x\in \mathbb{U}C(\Gamma)$, there exist $\{\alpha_k\}\subset\mathbb{U}$ and $\{T_k\}\subset\Gamma$ such that
$\alpha_k T_kx\longrightarrow y.$
Hence, $\Vert y \Vert\leq m$, which is a contradiction.
\end{proof}

We denote by $\{\Gamma\}^{'}$ the set of all elements of $L(X)$ which commutes with every element of $\Gamma.$
\begin{proposition}\label{p1}
Let $X$ be a complex topological vector space and  $\Gamma$ a subset of $L(X).$
Assume that $\Gamma$ is codiskcyclic and let $T\in L(X)$ be with dense range. If $T\in \{\Gamma\}^{'}$, then $Tx\in \mathbb{U}C(\Gamma)$, for all $x\in \mathbb{U}C(\Gamma)$.
\end{proposition}
\begin{proof}
Let $O$ be a nonempty open subset of $X$. Since $T$ is continuous and of dense range, $T^{-1}(O)$ is a nonempty open subset of $X$. Let $x\in \mathbb{U}C(\Gamma)$, then there exist $\alpha\in\mathbb{U}$ and $S\in \Gamma$  such that $\alpha Sx\in T^{-1}(O)$, that is $\alpha T(Sx)\in O$. Since $T\in \{\Gamma\}^{'}$, it follows that
$\alpha S(Tx)=\alpha T(S x)\in O.$
Hence, $\mathbb{U}Orb(\Gamma,Tx)$ meets every nonempty open subset of $X$. From this, $\mathbb{U}Orb(\Gamma,Tx)$ is dense in $X$. That is, $Tx\in \mathbb{U}C(\Gamma)$. 
\end{proof}


 Recall from \cite{AOC}, that $\Gamma\subset L(X)$ and $\Gamma_1\subset L(Y)$ are called quasi-similar if there exists a continuous map $\phi$ : $X\longrightarrow Y$ with dense range such that for all $T\in\Gamma,$ there exists $S\in\Gamma_1$ satisfying $S\circ\phi=\phi\circ T$. If $\phi$ can be chosen to be a
homeomorphism, then $\Gamma$ and $\Gamma_1$ are called similar.

In the following, we prove that the codiskcyclicity is preserved under quasi-similarity.
\begin{proposition}\label{14}
If $\Gamma$ and $\Gamma_1$ are quasi-similar, then $\Gamma$ is codiskcyclic in $X$ implies that $\Gamma_1$ is codiskcyclic in $Y$. Moreover, 
$\phi(\mathbb{U}C(\Gamma)\subset \mathbb{U}C(\Gamma_1).$
\end{proposition}
\begin{proof}
Assume that $\Gamma$ is codiskcyclic in $X$.
Let $O$ be a nonempty open subset of $Y$, then $\phi^{-1}(O)$ is a nonempty open subset of $X$. If $x\in \mathbb{U}C(\Gamma)$, then there exist $\alpha\in\mathbb{U}$ and $T\in\Gamma$ such that $\alpha Tx\in \phi^{-1}(O)$, that is $\alpha\phi(Tx)\in O$. Let $S\in\Gamma_1$ such that $S\circ\phi=\phi\circ T$. Hence,
$\alpha S(\phi x) =\alpha\phi(Tx)\in O.$
Hence, $\Gamma_1$ is codiskcyclic and $\phi x\in \mathbb{U}C(\Gamma_1)$.
\end{proof}
\begin{proposition}\label{prop2}
Let $(c_T)_{T\in \Gamma}\subset\mathbb{R}_{+}^{*}$. If $\{c_T T\mbox{ : }T\in \Gamma\}$ is codiskcyclic and $(k_T)_{T\in \Gamma}$ is such that $c_T\geq k_T>0$ for all $T\in\Gamma,$ then the set $\{k_T T\mbox{ : }T\in \Gamma\}$ is codiskcyclic.
\end{proposition}
\begin{proof}
Let $x$ be a codiskcyclic vector for $\{c_T T\mbox{ : }T\in \Gamma\}$. Since $c_T\geq k_T$ for all $T\in\Gamma$, we have
$ \mathbb{U} Orb(\{c_T T\mbox{ : }T\in \Gamma\},x)\subset \mathbb{U} Orb(\{k_T T\mbox{ : }T\in \Gamma\},x).$
Since $\mathbb{U} Orb(\{c_T T\mbox{ : }T\in \Gamma\},x)$ is dense in $X$, it follows that $\mathbb{U}Orb(\{k_T T\mbox{ : }T\in \Gamma\},x)$ is dense in $X$, this means that $\{k_T T\mbox{ : }T\in \Gamma\}$ is codiskcyclic in $X$.
\end{proof}


\begin{proposition}\label{p4}
Let $\{X_i\}_{i=1}^{n}$ be a family of complex topological vector spaces and $\Gamma_i$ a subset of
$L(X_i),$ for $1\leq i\leq n$. If $\bigoplus_{i=1}^n\Gamma_i$ is a codiskcyclic set in $\bigoplus_{i=1}^n X_i$, then $\Gamma_i$ is a codiskcyclic set in $X_i$, for all $1\leq i\leq n$.
\end{proposition}
\begin{proof}
If $1\leq j\leq n$, then $\bigoplus_{i=1}^n\Gamma_i$ is quasi-similar to $\Gamma_j$, and the result follow by Proposition \ref{14}.
\end{proof}


Let $X$ be a complex topological vector space.
The following proposition gives a characterization of the set of codiskcyclic vector of set of operators using a countable basis of the topology of $X$.
Note that the set $\mathbb{D}$ is the unit closed disk defined by
$\mathbb{D}=\{\alpha\in \mathbb{C} \mbox{ : }\vert \alpha \vert\leq 1\}.$
\begin{proposition}\label{p2}
Let $X$ be a second countable complex topological vector space and $\Gamma$ a subset of $L(X).$ If $\Gamma$is codiskcyclic, then
$$\mathbb{U}C(\Gamma)=\bigcap_{n\geq1}\left( \bigcup_{\beta\in\mathbb{D}}\bigcup_{T\in \Gamma} T^{-1}(\beta U_{n})\right),$$
where $(U_n)_{n\geq1}$ is a countable basis of the topology of $X$. As a consequence, $\mathbb{U}C(\Gamma)$ is a $G_\delta$ type set.
\end{proposition}
\begin{proof}
Let $x\in X$. Then, $x\in \mathbb{U}C(\Gamma)$ if and only if $\overline{\mathbb{U}Orb(\Gamma,x)}=X$. Equivalently, for all $n\geq 1,$ $U_n \cap \mathbb{U}Orb(\Gamma,x)\neq\emptyset,$ that is for all $n\geq 1$, there exist $\lambda\in\mathbb{U}$ and $T\in\Gamma$ such that $\lambda Tx\in U_n$. This is equivalent to the fact that for all $n\geq1,$ there exist $\beta\in\mathbb{D}$ and $T\in\Gamma$ such that $x\in T^{-1}(\beta U_n)$. Hence, $\displaystyle x\in \bigcap_{n\geq1} \bigcup_{\beta\in \mathbb{D}}\bigcup_{T\in \Gamma} T^{-1}(\beta U_{n}).$
\end{proof}
\section{Density and Codisk Transitivity of Sets of Operators}

In the following definition, we introduce the notion of a codisk transitive set of operators which generalize
the codisk transitivity of a single operator. 
\begin{definition}
We say that $\Gamma$ is a codisk transitive set of operators, if for any pair $(U,V)$ of nonempty open subsets of $X$, there exist $\alpha\in\mathbb{U}$ and $T\in\Gamma$ such that 
$T(\alpha U) \cap V\neq\emptyset.$
\end{definition}
\begin{remark}
An operator $T\in L(X)$ is codisk transitive if and only if
$\Gamma=\{T^n\mbox{ : }n\geq0\}$
is codisk transitive.
\end{remark}
\begin{example}\label{ex3}
Assume that $X$ is a locally convex space.
Let $x$, $y\in X$ and let $f_y$ be a linear form on $X$ such that $f_y(y)\neq0$. Let $T_{f_y,x}$ be an operator defined by
$T_{f_y,x}z = f_y(z)x.$
Define
$\Gamma=\{T_{f_y,x} \mbox{ : }x\mbox{, }y\in X \mbox{ such that }f_y(y)\neq0\}.$
Let $U$ and $V$ be two nonempty open subsets of $X$. There exist $x,$ $y\in X$ such that $x\in U$ and $y\in V$. We have
$ T_{f_y,x}(y)=f_y(y)x. $
Since $0<f_y(y)<1$, it follows that $x=\frac{1}{f_y(y)}T_{f_y,x}(y)$. Hence $x\in U$ and $x\in \frac{1}{f_y(y)}T_{f_y,x}(V)$, which implies that
$ U\cap \frac{1}{f_y(y)}T_{f_y,x}(V)\neq\emptyset. $
Thus $\Gamma$ is a codisk transitive.
\end{example}

In the following proposition, we prove that the codisk transitivity of sets of operators is preserved under quasi-similarity.
\begin{proposition}\label{prop1}
Assume that $\Gamma\subset L(X)$ and $\Gamma_1\subset L(Y)$ are quasi-similar. If $\Gamma$ is codisk transitive in $X$, then $\Gamma_1$ is codisk transitive in $Y$.
\end{proposition}
\begin{proof}
Let $U$ and $V$ be nonempty open subsets of $X$.
Since $\phi$ is continuous and of dense range, $\phi^{-1}(U)$ and $\phi^{-1}(V)$ are nonempty and open sets.
Since $\Gamma$ is codisk transitive in $X$, there exist  $y\in \phi^{-1}(U)$ and $\alpha\in\mathbb{U},$ $T\in\Gamma$ with $\alpha Ty\in\phi^{-1}(V)$, which implies that $\phi(y)\in U$ and $\alpha \phi(Ty)\in V$. Let $S\in\Gamma$ such that $S\circ\phi=\phi\circ T$. Then, $\phi(y)\in U$ and $\alpha S\phi(y)\in V$. Thus, $\alpha S(U)\cap V\neq\emptyset.$ Hence, $\Gamma_1$ is codisk transitive in $Y.$
\end{proof}

In the following result, we give necessary and sufficient conditions for a set of operators to be codisk transitive.
\begin{theorem}\label{tt}
Let $X$ be a complex normed space and $\Gamma$ a subset of $L(X)$.
The following assertions are equivalent$:$ 
\begin{itemize}
\item[$(i)$] $\Gamma$ is codisk transitive;
\item[$(ii)$] For each $x$, $y\in X,$ there exists sequences $\{x_k\}$ in $X$, $\{\alpha_k\}$ in $\mathbb{U}$ and $\{T_k\}$ in $\Gamma$ such that
$x_k\longrightarrow x$ and $\alpha_k T_k( x_k)\longrightarrow y;$
\item[$(iii)$] For each $x$, $y\in X$ and for $W$ a neighborhood of $0$, there exist $z\in X$, $\alpha\in\mathbb{U}$ and $T\in\Gamma$  such that 
$x-z\in W $ and $ \alpha T(z)-y\in W. $ 
\end{itemize}
\end{theorem}
\begin{proof}$(i)\Rightarrow(ii)$
\noindent Let $x$, $y\in X$. For all $k\geq1$, let $U_k=B(x,\frac{1}{k})$ and $V_k=B(y,\frac{1}{k})$. Then $U_k$ and $V_k$ are nonempty open subsets of $X$. Since $\Gamma$ is codisk transitive, there exist $\alpha_k\in\mathbb{U}$ and $T_k\in\Gamma$ such that $\alpha_k T_k( U_k)\cap V_k\neq\emptyset$. For all $k\geq1$, let $x_k\in U_k$ such that $ \alpha_k T_k( x_k)\in V_k$, then 
$\Vert x_k-x \Vert<\frac{1}{k}$ and $\Vert \alpha_k T_k( x_k)-y \Vert<\frac{1}{k},$
this implies that 
$x_k\longrightarrow x\hspace{0.3cm}\mbox{ and }\hspace{0.3cm} \alpha_k T_k( x_k)\longrightarrow y.$

\noindent $(ii)\Rightarrow(iii)$ Clear.

\noindent $(iii)\Rightarrow(i)$ Let $U$ and $V$ be two nonempty open subsets of $X$. Then there exists $x$, $y\in X$ such that $x\in U$ and $y\in V$. Since for all $k\geq1$,  $W_k=B(0,\frac{1}{k})$ is a neighborhood of $0$,  there exist $z_k\in X,$ $\alpha_k\in\mathbb{U}$ and $T_k\in\Gamma$ such that 
$\Vert x-z_k\Vert<\frac{1}{k}$ and $\alpha_k T_k(z_k)-y\Vert<\frac{1}{k}.$
This implies that 
$z_k\longrightarrow x$ and $\alpha_k T_k(z_k)\longrightarrow y.$
Since $U$ and $V$ are nonempty open subsets of $X$, $x\in U$ and $y\in V$, there exists $N\in\mathbb{N}$ such that $z_k\in U$ and $\alpha_k T_k(z_k)\in V$, for all $k\geq N.$ 
\end{proof}


Let $\Gamma$ be a subset of $L(X)$.
In whats follows, we prove that $\Gamma$ is codisk transitive if and only if it admits a dense subset of codiskcyclic vectors.
\begin{theorem}\label{t1}
Let $X$ be a second countable Baire complex topological vector space and $\Gamma$ a subset of $L(X)$. The following assertions are equivalent$:$
\begin{itemize}
\item[$(i)$] $\mathbb{U}C(\Gamma)$ is dense in $X$;
\item[$(ii)$] $\Gamma$ is codisk transitive.
\end{itemize}
As a consequence, a codisk transitive set is codiskcyclic.
\end{theorem}
\begin{proof} 
$(i)\Rightarrow (ii) :$ Assume that $\mathbb{U}C(\Gamma)$ is dense in $X$ and let $U$ and $V$ be two nonempty open subsets of $X$.  By Proposition \ref{p2}, we have 
$\mathbb{U}C(\Gamma)=\bigcap_{n\geq1}\left(\bigcup_{\beta\in\mathbb{D}}\bigcup_{T\in \Gamma} T^{-1}(\beta U_{n})\right).$
Hence, for all $n\geq 1,$ $\displaystyle A_n:=\bigcup_{\beta\in\mathbb{D}}\bigcup_{T\in \Gamma} T^{-1}(\beta U_{n})$ is dense in $X$. Thus, for all $n,$ $m\geq 1,$ we have $A_n\cap U_m\neq\emptyset$ which implies that for all $n,$ $m\geq 1,$ there exist $\beta\in\mathbb{U}$ and $T\in \Gamma$ such that $T(\beta U_m)\cap U_n\neq\emptyset$.
Hence, $\Gamma$ is a codisk transitive set.\\
$(ii)\Rightarrow (i) : $ Assume that $\Gamma$ is codisk transitive. Let $n$, $m\geq 1$, then there exist
$\beta\in \mathbb{U}$ and $T\in \Gamma$ such that $T(\beta U_m)\cap U_n\neq \emptyset$,
which implies that $T^{-1}(\frac{1}{\beta} U_n)\cap U_m\neq \emptyset$. Hence, for all $n\geq1,$ we have
$\displaystyle\bigcup_{\beta\in\mathbb{D}}\bigcup_{T\in \Gamma} T^{-1}(\beta U_{n})$ is dense in $X$.
Since $X$ is a Baire space, it follows that 
$\mathbb{U}C(\Gamma)=\bigcap_{n\geq1}\left(\bigcup_{\beta\in\mathbb{D}}\bigcup_{T\in \Gamma} T^{-1}(\beta U_{n})\right)$
is a dense subset of $X$.
\end{proof}

The converse of Theorem \ref{t1} holds with some additional assumption.
\begin{theorem}\label{so}
Let $X$ be a complex topological vector space and $\Gamma$ a subset of $L(X)$.
Assume that for all $T$, $S\in\Gamma$ with $T\neq S$, there exists $A\in\Gamma$ such that $T=AS$. The following assertions are equivalent$:$
\begin{itemize}
\item[$(i)$] $\Gamma$ is codiskcyclic;
\item[$(ii)$] $\Gamma$ is codisk transitive.
\end{itemize}
\end{theorem}
\begin{proof}
 $(i)\Rightarrow(ii)$ This implication is due to Theorem \ref{t1}.\\
$(ii)\Rightarrow(i)$ Since $\Gamma$ is codiskcyclic, there exists $x\in X$ such that $\mathbb{U}Orb(\Gamma,x)$ is a dense subset of $X$.
Let $U$ and $V$ be two nonempty open subsets of $X$, then there exist $\alpha$, $\beta\in\mathbb{U}$ with $\vert\alpha\vert\geq\vert\beta\vert$, and $T$, $S\in\Gamma$ such that
$\alpha Tx\in U$ and $\beta Sx\in V. $
There exists $A\in\Gamma$ such that $T=AS$.
Hence,
$\alpha A(Sx)\in U$ and $\beta A(Sx)\in  A(V)$,
which implies that $U\cap A(\frac{\alpha}{\beta} V)\neq \emptyset$. 
Hence, $\Gamma$ is codisk transitive.
\end{proof}

In the following definition we introduce the notion of strictly codisk transitivity of a set of
operators. The case of hypercyclicity (resp, supercyclicity, diskcyclicity) were introduced in \cite{AOD,AOS,AKH}.
\begin{definition}
We say that $\Gamma$ is strictly codisk transitive if for each pair of nonzero elements
$x,$ $y$ in $X$, there exist some $\alpha\in\mathbb{U}$ and $T\in\Gamma$ such that 
$\alpha Tx=y.$
\end{definition}
\begin{remark}

An operator $T\in L(X)$ is strictly codisk transitive if and only if the set
$\Gamma=\{T^n\mbox{ : }n\geq0\} $
is a strictly codisk transitive.
\end{remark}
\begin{proposition}
If $\Gamma$
is strictly codisk transitive set, then it is codisk transitive. As a consequence, if $\Gamma$ is strictly codisk transitive set, then it is codiskcyclic.
\end{proposition}
\begin{proof}
Assume that $\Gamma$ is a strictly codisk transitive set. If $U$ and $V$ are two nonempty open subsets of $X$, then there exist $x,$  $y\in X$ such that
$x\in U$ and $y\in V$. Since $\Gamma$ is strictly codisk transitive, it follows that there exist $\alpha\in\mathbb{U}$ and $T\in\Gamma$ such that $\alpha Tx=y.$ Hence,
$ \alpha Tx\in \alpha T(U)$ and $\alpha Tx\in V. $
Thus, $\alpha T(U)\cap V\neq\emptyset,$ which implies that $\Gamma$ is codisk transitive. By Theorem \ref{t1}, we deduce that $\Gamma$ is codiskcyclic.
\end{proof}

In the following proposition, we prove that the strictly codisk transitivity of sets of operators is preserved under similarity.
\begin{proposition}
If $\Gamma\subset L(X)$ and $\Gamma_1 L(Y)$ are similar, then $\Gamma$ is strictly codisk transitive in $X$ if and only if  $\Gamma_1$ is strictly codisk transitive in $Y.$
\end{proposition}
\begin{proof}
Let $x$, $y\in Y$. There exist $a,$ $b\in X$ such that $\phi(a)=x$ and $\phi(b)=y$. Since $\Gamma$ is strictly codisk transitive in $X$, there exist $\alpha\in\mathbb{U}$ and $T\in \Gamma$ such that $\alpha Ta=b,$ this implies that $\alpha \phi\circ T(a)=\phi(b)$. Let $S\in\Gamma_1$ such that $S\circ\phi=\phi\circ T$. Hence, $\alpha Sx=y$. Hence $\Gamma_1$
is strictly codisk transitive in $Y$.
\end{proof}
Recall that the strong operator topology (SOT for short) on $L(X)$ is the topology with respect to which any $T\in L(X)$ has a neighborhood basis consisting of sets of the form
$$\Omega=\{S\in L(X) \mbox{ : }Se_i-Te_i\in U\mbox{, }i=1,2,\dots,k\},$$
 where $k\in\mathbb{N}$, $e_1,e_2,\dots e_k\in X$ are linearly independent and $U$ is a neighborhood of zero in $X$, see \cite{Conway}.
 
Let $x$ be an element of a complex topological vector space $X$. Note that $\mathbb{U}_x$ is the subset of $X$ defined by 
$\mathbb{U}\{x\}:=\mathbb{U}_x= \{\alpha x\mbox{ : }\alpha\in\mathbb{U}\}. $

 In the following theorem, the proof is also true for norm-density if $X$ is assumed to be a normed linear space. 
\begin{theorem}
For each pair of nonzero vectors $x$, $y\in X$ with $y\notin \mathbb{U}_x$, there exists
a SOT-dense set $\Gamma_{xy}\subset L(X)$ which is not strictly codisk transitive. Furthermore, $\Gamma\subset L(X)$ is a dense nonstrictly
codisk transitive set if and only if $\Gamma$ is a dense subset of $\Gamma_{xy}$ for some $x$, $y\in X.$
\end{theorem}
\begin{proof}
Fix nonzero vectors $x,$ $y\in X$ such that $y\notin\mathbb{U}_x$ and let $\Gamma_{xy}$ the set defined by  
$$\Gamma_{xy}=\{T\in L(X) \mbox{ : }y\notin\mathbb{U}_{Tx}\}.$$
Then $\Gamma_{xy}$ is not strictly codisk transitive. Let $\Omega$ be a nonempty
open set in $L(X)$ and $S\in\Omega$. If $Sx$ and $y$ are such that
$y\notin\mathbb{U}_{Sx}$, then $S\in\Omega\cap\Gamma_{xy}$.
Otherwise, putting $S_n = S+\frac{1}{n}I$, we see that $S_k\in\Omega$ for some $k$, but $S_kx$ and $y$ are such that
$y\notin\mathbb{U}_{S_kx}$. Hence, $\Omega\cap\Gamma_{xy}\neq\emptyset$ and the proof is completed.

We prove the second assertion of the theorem. Suppose that $\Gamma$ is a dense subset of $L(X)$ that is not
strictly codisk transitive. Then there are nonzero vectors $x,$ $y\in X$ such that $y\notin\mathbb{U}_{Tx}$
for all $T\in \Gamma$ and hence $\Gamma\subset \Gamma_{xy}$.
To show that $\Gamma$ is dense in $\Gamma_{xy}$, assume that $\Omega_0$ is an open subset of 
$\Gamma_{xy}$. Thus, $\Omega_0= \Gamma_{xy}\cap\Omega$ for some open
set $\Omega$ in $L(X)$. Then $\Gamma\cap \Omega_0= \Gamma\cap \Omega\neq\emptyset$.

For the converse, let $\Gamma$ be a dense subset of $\Gamma_{xy}$ for some $x,$ $y\in X$. Then $\Gamma$ is not strictly codisk transitive. Also, since $\Gamma_{xy}$ is a dense open subset of $L(X)$, we conclude that $\Gamma$ is also dense in $L(X)$. Indeed,
if $\Omega$ is any open set in $L(X)$ then $\Omega\cap\Gamma_{xy}\neq\emptyset$ since $\Gamma_{xy}$ is dense in $L(X)$. On the other hand, $\Omega\cap\Gamma_{xy}$
is open in $\Gamma_{xy}$ and so it must intersect $\Gamma$ since $\Gamma$ is dense in $\Gamma_{xy}$. Thus,  
$\Omega\cap\Gamma\neq\emptyset$ and so $\Gamma$ is dense in $L(X)$.
\end{proof}
\begin{corollary}
There is a subset $\Gamma_1$
of $\Gamma$ such that $\overline{\Gamma_1}= L(X)$ and $\Gamma_1$ is not strictly codisk transitive.
\end{corollary}
\begin{proof}
For nonzero $x,$ $y$ such that $y\notin\mathbb{U}_x,$ put $\Gamma_1=\Gamma\cap\Gamma_{xy}$.
\end{proof}

In the following definition, we introduce that notion of codiskcyclic transitivity of set of operators.
The case of hypercyclicity (resp, supercyclicity, diskcyclicity) were introduced in \cite{AOD,AOS,AKH}.
\begin{definition}
We say that $\Gamma$ is a codiskcyclic transitive set or codiskcyclic transitive if 
$\mathbb{U}C(\Gamma)=X\setminus\{0\}.$
\end{definition}
\begin{remark}
An operator $T\in L(X)$ is codiskcyclic transitive if and only if the set
$\Gamma=\{T^n\mbox{ : }n\geq0\} $
is codiskcyclic transitive.
\end{remark}

It is clear that a codiskcyclic transitive set is codiskcyclic. Moreover, the next proposition shows that codiskcyclic transitivity of sets of operators implies codisk transitivity.
\begin{proposition}
If $\Gamma$ is codiskcyclic transitive, then $\Gamma$ is codisk transitive. 
\end{proposition}
\begin{proof}
Let $U$ and $V$ be two nonempty open subsets of $X$. There exists $x\in X\setminus\{0\}$ such that $x\in U$. Since $\Gamma$ is codiskcyclic transitive, there exists $\alpha\in\mathbb{U}$ and $T\in\Gamma$ such that $\alpha Tx\in V$. This implies that $\alpha T(U)\cap V\neq\emptyset.$ Hence, $\Gamma$ is codisk transitive.
\end{proof}

In the following proposition, we prove that the codiskcyclic transitivity is preserved under similarity.
\begin{proposition}
Assume that $\Gamma$ and $\Gamma_1$ are similar, then $\Gamma$ is codiskcyclic transitive on $X$ if and only if $\Gamma_1$ is codiskcyclic transitive on $Y$.
\end{proposition}
\begin{proof}
If $\Gamma$ is a codiskcyclic transitive on $X$, then by Proposition \ref{14}, $\phi(\mathbb{U}C(\Gamma))\subset \mathbb{U}C(\Gamma_1)$. Since $\phi$ is homeomorphism, the result holds.
\end{proof}

Assume that $X$ is a topological vector space and $\Gamma$ a subset of $L(X)$.
The following result shows that the SOT-closure of $\Gamma$ is not large enough to have more codiskcyclic vectors than $\Gamma$.
\begin{proposition}\label{prop3}
If $\overline{\Gamma}$ stands for the SOT-closure of $\Gamma$ then 
$\mathbb{U}C(\Gamma)=\mathbb{U}C(\overline{\Gamma}).$
\end{proposition}
\begin{proof}
We only need to prove that $\mathbb{U}C(\overline{\Gamma})\subset \mathbb{U}C(\Gamma)$. Fix $x\in \mathbb{U}C(\overline{\Gamma})$ and let $U$ be an arbitrary open subset of $X$. Then there is some $\alpha\in\mathbb{U}$ and $T\in\overline{\Gamma}$ such that $\alpha Tx \in U$. The set $\Omega=\{S\in L(X) \mbox{ : } \alpha Sx \in U\}$ is a SOT-neighborhood
of $T$ and so it must intersect $\Gamma$. Therefore, there is some $S\in \Gamma$ such that $\alpha Sx \in U$ and this shows that
$x\in \mathbb{U}C(\Gamma)$.
\end{proof}
\begin{corollary}
 Let X be a topological vector space and $\Gamma$ a subset of $L(X)$. Then $\Gamma$ is 
 codiskcyclic transitive if and only if $\overline{\Gamma}$ is codiskcyclic transitive.
\end{corollary}
\begin{proof}
Assume that $\Gamma$ is codiskcyclic transitive, then $\mathbb{U}C(\overline{\Gamma})=X\setminus\{0\}$. 
Since by Proposition \ref{prop3},
we have $\mathbb{U}C(\overline{\Gamma})=\mathbb{U}C(\Gamma)$, 
it follows that $\mathbb{U}C(\Gamma)=X\setminus\{0\}$. Hence, $\Gamma$ is codiskcyclic transitive.
\end{proof}

In the next definition, we introduce the notion of codiskcyclic criterion of a set of operators which generalizes
the definition of codiskcyclic criterion of a single operator.
\begin{definition}\label{cc}
We say that $\Gamma$ satisfies the criterion of codiskcyclicity if there exist two dense subsets $X_0$ and 
$Y_0$ in $X$ and sequences $\{\alpha_k\}$ of $\mathbb{U}$, $\{T_k\}$ of $
\Gamma$ and a sequence of maps $S_k$ : $Y_0\longrightarrow X$ such that$:$
\begin{itemize}
\item[$(i)$] $\alpha_k T_kx\longrightarrow 0$ for all $x\in X_0$;
\item[$(ii)$] $\alpha_k^{-1} S_kx\longrightarrow 0$ for all $y\in Y_0$;
\item[$(iii)$] $T_kS_ky\longrightarrow y$ for all $y\in Y_0$.
\end{itemize}
\end{definition}
\begin{remark}
An operator $T\in L(X)$ satisfies the criterion of codiskcyclicity for operators if and only if the set
$\Gamma=\{T^n \mbox{ : }n\geq0\} $
satisfies the criterion of codiskcyclicity for sets of operators, see \cite{Zeana}.
\end{remark}
\begin{theorem}\label{11}
Let $X$ be a second countable Baire complex topological vector space and $\Gamma$ a subset of $L(X)$. If $\Gamma$ satisfies the criterion of codiskcyclicity, then $\mathbb{U}C(\Gamma)$ is a dense subset of $X$. As consequence; $\Gamma$ is codiskcyclic.
\end{theorem}
\begin{proof}
Let $U$ and $V$ be two nonempty open subsets of $X$. Since $X_0$ and $Y_0$ are dense in $X$, there exist $x_0$ and $y_0$ in $X$ such that
$ x_0\in X_0\cap U$ and $y_0\in Y_0\cap V.$
For all $k\geq1,$ let $z_k=x_0+\alpha_k^{-1} S_ky $. We have 
$\alpha_k^{-1} S_k y\longrightarrow 0$, which implies that $z_k\longrightarrow x_0$. Since $x_0\in U$
and $U$ is open, there exists $N_1\in\mathbb{N}$ such that $z_k\in U$, for all $k\geq N_1$. On the
other hand, we have $\alpha_k T_k z_k=\alpha_k T_k x_0+T_k (S_k y_0)\longrightarrow y_0$. Since $y_0\in V$
and $V$ is open, there exists $N_2\in\mathbb{N}$ such that $\alpha_k T_k z_k\in V$,
for all $k\geq N_2$. Let $N=$max$\{N_1,N_2\}$, then  $z_k\in U$ and $\alpha_k T_k z_k\in V$,
for all $k\geq N$, that is 
$\alpha_k T_k (U)\cap V\neq \emptyset,$
for all $k\geq N$.
Hence, $\Gamma$ is codisk transitive. By Theorem \ref{t1} we deduce that $\mathbb{U}C(\Gamma)$ is a dense subset of $X$. We use again Theorem \ref{t1} to conclude that $\Gamma$ is codiskcyclic and this complete the proof.
\end{proof}

\section{Codiskcyclic $C_0$-Semigroups of Operators}
In this section we will study the particular case when $\Gamma$ is a $C_0$-semigroup of operators.

Recall that a family $(T_t)_{t\geq0}$ of operators is called a $C_0$-semigroup of operators if the following three conditions are satisfied$:$
\begin{itemize}
\item[$(i)$] $T_0=I$ the identity operator on $X$;
\item[$(ii)$] $T_{t+s}=T_{t}T_{s}$ for all $t,$ $s\geq0$;
\item[$(iii)$] $\lim_{t\rightarrow s}T_{t}x=T_{s}x$ for all $x\in X$ and $t\geq 0$.
\end{itemize}
For more informations about the theory of $C_0$-semigroups the reader may refer to \cite{Pazy}.

\begin{example}\label{exc}
Let $X=\mathbb{C}$. For all $t\geq0$, let $T_tx= \exp(t)x$, for all $x\in \mathbb{C}$.
Then $(T_t)_{t\geq0}$ is a $C_0$-semigroup and we have
$ \mathbb{U}Orb((T_t)_{t\geq0},1)=\{\alpha T_t(1)\mbox{ : }t\geq0\mbox{, }\alpha\in\mathbb{U}\}=\{\alpha y \mbox{ : }y\in\mathbb{R}^+\mbox{, }\alpha\in\mathbb{U}\}.$
Let $x\in\mathbb{C}\setminus\{0\}$. Then $\displaystyle x=\vert x\vert\frac{x}{\vert x\vert}\in \mathbb{U}Orb((T_t)_{t\geq0},1)$.  Hence,
$\overline{\mathbb{U}Orb((T_t)_{t\geq0},1)}=\mathbb{C}.$
Thus, $(T_t)_{t\geq0}$ is a codiskcyclic $C_0$-semigroup of operators and $1$ is a codiskcyclic vector for $(T_t)_{t\geq0}$.
\end{example}

Recall from \cite[Lemma 5.1]{Wengenroth}, that if $X$ is a complex topological vector space such that $2 \leq$ dim$(X) <\infty$, then $X$ supports no supercyclic $C_0$-semigroups of operators.

In the following theorem we will prove that the same result holds in the case of codiskcyclicity on a complex topological vector space.
\begin{theorem}
Assume that $2 \leq$ dim$(X) <\infty$. Then $X$ supports no codiskcyclic $C_0$-semigroups.
\end{theorem}
\begin{proof}
By using \cite[Lemma 5.1]{Wengenroth} and the fact that
$\mathbb{U}Orb(\Gamma,x)\subset \mathbb{C}Orb(\Gamma,x).$
\end{proof}

A necessary and sufficient condition for a $C_0$-semigroup of operators to be
codiskcyclic is given in the next lemma and theorem.
\begin{lemma}\label{lm1}
Let $(T_t )_{t\geq0}$ be a codiskcyclic $C_0$-semigroup of operators on a Banach infinite dimensional space $X$. If $x \in X$ is a codiskcyclic vector of $(T_t )_{t\geq0}$, then the following assertions hold:
\begin{itemize}
\item[$(1)$] $T_t x\neq 0$, for all $t\geq 0$;
\item[$(2)$] The set $\{\alpha T_t x \mbox{ : }t\geq s\mbox{, }\alpha\in\mathbb{U}\}$ is dense in $X$, for all $s\geq 0$.
\end{itemize}
\end{lemma}
\begin{proof}
$(1)$ Suppose that $t_0>0$ is minimal with the property that $T_{t_0}x=0$. We show first that each $y \in X$ is of the form $y = \alpha T_tx$ for some $t\in[0,t_0]$ and $\alpha\in\mathbb{U}$. Since $x\in\mathbb{U}C(\Gamma)$, there exist a sequence $(t_n)_{n\in\mathbb{N}}\subset [0,t_0]$ and a sequence $(\alpha_n)_{n\in\mathbb{N}}\subset \mathbb{U}$ such that $\alpha_n T_{t_n}x\longrightarrow y$. Without loss of generality we may assume that $(t_n)_{n\in\mathbb{N}}$ converges to some $t$.
By compactness we may assume that $(\alpha_n)_{n\in\mathbb{N}}$ converges to some $\alpha$ and we infer that $y=\alpha T_t x$.

Now take three vectors $y_i=\alpha_i T_{t_i}x\in X$, spanning a two-dimensional subspace,
such that each pair $y_i$, $y_j$, $i \neq j$, is linearly independent. Assume that $t_1 > t_2 > t_3$.
We have then $y_3 = c_1 y_1 + c_2 y_2$. Now we arrive at the contradiction
\begin{align*}
0&\neq \alpha_3 T_{(t_0+t_3-t_2)}x=T_{(t_0-t_2)}y_3=c_1T_{(t_0-t_2)}y_1+c_2T_{(t_0-t_2)}y_2\\
 &=c_1 \alpha_1 T_{(t_0+t_1-t_2)}x+c_2\alpha_2 T_{t_0}x=0.
\end{align*}
$(2)$ Suppose that there exists $s_0>0$ such that $\{\alpha T_t x \mbox{ : }t\geq s_0\mbox{, }\alpha\in\mathbb{U}\}$ is not dense in X. Hence there exists a bounded open set $U$ such that $U \cap \overline{A}= \emptyset$.
Therefore we have
$ U\subset \overline{\{\alpha T_t x \mbox{ : }0\leq t\leq s_0\mbox{, }\alpha\in\mathbb{U}\}} $
by using the relation 
$$ X=\overline{\{\alpha T_t x \mbox{ : }t\geq 0\mbox{, }\alpha\in\mathbb{U}\}}=\overline{\{\alpha T_t x \mbox{ : }t\geq s_0\mbox{, }\alpha\in\mathbb{U}\}}\cup \overline{\{\alpha T_t x \mbox{ : }0\leq t\leq s_0\mbox{, }\alpha\in\mathbb{U}\}}. $$
Thus, $\overline{U}$ is compact. Hence $X$ is finite dimensional, which contradicts that $X$ is infinite dimensional.
\end{proof}
\begin{theorem}\label{t6}
Let $(T_t)_{t\geq0}$ be a $C_0$-semigroup of operators on a separable Banach infinite dimensional space X. Then the following assertions are equivalent:
\begin{itemize}
\item[$(1)$]$(T_t)_{t\geq0}$ is codiskcyclic;
\item[$(2)$]for all $y$, $z \in X$ and all $\varepsilon > 0$, there exist $v \in X$, $t>0$ and $\alpha \in\mathbb{U}$ such that
$\Vert y-v \Vert<\varepsilon$ and $\Vert z-\alpha T_tv \Vert<\varepsilon;$
\item[$(3)$]for all $y$, $z \in X$, all $\varepsilon > 0$ and for all $l\geq 0$,  there exist $v \in X$, $t>l$ and $\alpha \in\mathbb{U}$ such that
$\Vert y-v \Vert<\varepsilon$ and $\Vert z-\alpha T_tv \Vert<\varepsilon.$
\end{itemize}
\end{theorem}
\begin{proof}
$(1)\Rightarrow (3)$: Let  $x \in X$ such that $ \{ \alpha T_tx \mbox{ : }t\geq0 \mbox{, }\alpha\in\mathbb{U} \} $ is dense in $X$ and let $\varepsilon>0$. For any $y\in X$, there exist $s_1>0$ and $\alpha_1\in \mathbb{U}$ such that $\Vert y-\alpha_1 T_{s_1}x\Vert<\varepsilon $. If $l\geq0$, then by Lemma \ref{lm1}, the set  
$\alpha_1 \{ \alpha T_tx \mbox{ : }t\geq s+l \mbox{, }\alpha\in\mathbb{U} \}:=\{\alpha_1 \alpha T_tx \mbox{ : }t\geq s+l \mbox{, }\alpha\in\mathbb{U} \}$
 is dense in $X$. For any $z\in X$, there exist $s_2>l+s_1$ and $\alpha_2\in\mathbb{U}$ such that $\Vert z-\alpha_1\alpha_2 T_{s_2}x \Vert<\varepsilon$. Put $v=\alpha_1 T_{s_1}x$, $t=s_2-s_1>l$ and $\alpha=\alpha_2$. Then we have
$\Vert y-v \Vert<\varepsilon$ and $\Vert z-\alpha T_tv \Vert<\varepsilon.$\\
$(3)\Rightarrow (2)$: It is obvious.\\
$(2)\Rightarrow (1)$: Let $\{z_1,z_2,z_3,...\}$ be a dense sequence in $X$. we construct sequences $\{y_1,y_2,y_3,...\}\subset X$, $\{t_1,t_2,t_3,...\}\subset [0,+\infty)$ and $\{\alpha_1,\alpha_2,\alpha_3,...\}\subset \mathbb{U}$ inductively:
\begin{itemize}
\item Put $y_1=z_1$, $t_1=0$;
\item For $n>1$, find $y_n$, $t_n$ and $\alpha_n$ such that
\begin{equation}\label{e1}
\Vert y_n- y_{n-1} \Vert\leq \frac{2^{-n}}{\sup\{ \Vert T_{t_j} \Vert \mbox{ : } j<n \}},
\end{equation}
and 
\begin{equation}\label{e2}
\Vert z_n-\alpha_n T_{t_n}y_n \Vert\leq 2^{-n}.
\end{equation}
\end{itemize}
In particular, (\ref{e1}) implies that $\Vert y_n- y_{n-1} \Vert\leq 2^{-n}$, so that the sequence $(y_n)_{n\geq1}$ has a limit $x$. Applying (\ref{e2}) and once again (\ref{e1}) we infer that
\begin{align*}
\Vert z_n-\alpha_nT_{t_n}x \Vert&=\Vert z_n-\alpha_nT_{t_n}y_n+\alpha_nT_{t_n}y_n-\alpha_nT_{t_n}x  \Vert\\
                        &\leq \Vert z_n-\alpha_nT_{t_n}y_n\Vert+\Vert \alpha_nT_{t_n}(y_n-x)  \Vert\\
                        &\leq \Vert z_n-\alpha_nT_{t_n}y_n\Vert+\Vert \alpha_nT_{t_n}\Vert \sum_{i=n+1}^{+\infty} \Vert  y_i-y_{i-1}  \Vert\\
                           &\leq 2^{-n}+\sum_{i=n+1}^{+\infty}2^{-i}=2^{-n+1}.
\end{align*}
Given $z\in X$ and $\varepsilon>0$ there are arbitrarily large $n$ such that $\Vert z_n-z \Vert<\frac{\varepsilon}{2}$. Choosing $n$ large enough such that $2^{-n+1}<\frac{\varepsilon}{2}$, we obtain
$\Vert \alpha_n T_{t_n}x-z \Vert\leq \Vert z-z_n \Vert+\Vert z_n -\alpha_n T_{t_n}x\Vert<\varepsilon.$
Therefore, $\{ \alpha T_tx \mbox{ : }t\geq0 \mbox{, }\alpha\in\mathbb{U} \}$ is dense in $X$.
\end{proof}
As a corollary we obtain a sufficient condition of codiskcyclicity of a $C_0$-semigroup of operators.

Let $X$ be a separable Banach infinite dimensional space. Denote $ X_0$ the set of all $x\in X $ such that $\lim_{t\longrightarrow \infty }T_t x=0, $
and $ X_\infty$ the set of all $ x \in X$ such that for each  $\varepsilon > 0$ there exist some $w \in X$ , $\alpha\in\mathbb{U}$ and some $ t > 0$ with $ \Vert w \Vert < \varepsilon$ and $ \Vert\alpha T_t w -x \Vert < \varepsilon.$
\begin{theorem}
Let $(T_t)_{t\geq0}$ be a $C_0$-semigroup of operators on a separable Banach infinite dimensional space $X$. If both $X_\infty$ and $X_0$ are dense subsets, then $(T_t)_{t\geq0}$ is codiskcyclic.
\end{theorem}
\begin{proof}
Let $z \in X_\infty$ and $y \in X_0$. Then for each $\varepsilon > 0$ there are
arbitrarily large $t > 0$, $\alpha\in\mathbb{U}$ and $w \in X$ such that
$ \Vert w \Vert < \varepsilon$ and $\Vert\alpha T_t w -x \Vert < \frac{\varepsilon}{2}.$
Since $y \in X_0$, for sufficiently large $t$ we have $\Vert\alpha T_{t}y \Vert < \frac{\varepsilon}{2}$. We put $v=y+w$ and infer
$ \Vert z-T_t v \Vert\leq \Vert z-T_t w \Vert+\Vert\alpha T_t y \Vert<\varepsilon, $
and 
$ \Vert y-v \Vert=\Vert w \Vert<\varepsilon. $
By Theorem \ref{t6}, the result holds.
\end{proof}

We use Theorem \ref{so} to prove that the codiskcyclicity and  codisk transitivity of a $C_0$-semigroup of operators on a complex topological vector space are equivalent.

\begin{theorem}\label{23}
Let $(T_t)_{t\geq0}$ be a $C_0$-semigroup of operators on a complex topological vector space $X$. Then, the following assertions are equivalent$:$
\begin{itemize}
\item[$(i)$] $(T_t)_{t\geq0}$ is codiskcyclic;
\item[$(ii)$] $(T_t)_{t\geq0}$ is codisk transitive.
\end{itemize}
\end{theorem}
\begin{proof}
By remarking that if $t_1>t_2\geq0$, then there exists $t=t_1-t_2$ such that $T_{t_1}=T_t T_{t_2}$, and using Theorem \ref{so}.
\end{proof}

\end{document}